\documentclass[oneside]{amsart}
\usepackage[utf8]{inputenc}
\usepackage{amssymb}
\usepackage{graphicx}
\usepackage{geometry}
\usepackage{enumerate}
\usepackage{hyperref}

\newtheorem{lemma}{Lemma}[section]
\newtheorem{theorem}[lemma]{Theorem}

\newtheorem{conjecture}[lemma]{Conjecture}
\newtheorem{coro}[lemma]{Corollary}

\DeclareMathOperator{\conv}{conv}

\newcommand{\R}{\mathbb{R}}
\newcommand{\F}{\mathcal{F}}

\newcommand{\K}{\mathcal{K}}

\title{Transversals to Colorful Intersecting Convex Sets}

\author[C. Gomez-Navarro]{Cuauhtemoc Gomez-Navarro}
\address[C. Gomez-Navarro]{Facultad de Ciencias, UNAM, Ciudad de Mexico, Mexico}
\email{cgn@ciencias.unam.mx}

\author[E. Roldán-Pensado]{Edgardo Roldán-Pensado}
\address[E. Roldán-Pensado]{Centro de Ciencias Matemáticas, UNAM Campus Morelia, Morelia, Mexico}
\email{e.roldan@im.unam.mx}

\begin{document}
	
	\begin{abstract}
		Let $K$ be a compact convex set in $\R^{2}$ and let $\F_{1}, \F_{2}, \F_{3}$ be finite families of translates of $K$ such that $A \cap B \neq \emptyset$ for every $A \in \F_{i}$ and $B \in \F_{j}$ with $i \neq j$.
		A conjecture by Dol’nikov is that, under these conditions, there is always some $j \in \{ 1,2,3 \}$ such that $\F_{j}$ can be pierced by $3$ points. In this paper we prove a stronger version of this conjecture when $K$ is a body of constant width or when it is close in Banach-Mazur distance to a disk. We also show that the conjecture is true with $8$ piercing points instead of $3$. Along the way we prove more general statements both in the plane and in higher dimensions.
		
		A related result was given by Martínez-Sandoval, Roldán-Pensado and Rubin. They showed that if $\mathcal{F}_{1}, \dots, \mathcal{F}_{d}$ are finite families of convex sets in $\mathbb{R}^{d}$ such that for every choice of sets $C_{1} \in \F_{1}, \dots, C_{d} \in \F_{d}$ the intersection $\bigcap_{i=1}^{d} {C_{i}}$ is non-empty, then either there exists $j \in \{ 1,2, \dots, n \}$ such that $\mathcal{F}_j$ can be pierced by few points or $\bigcup_{i=1}^{n} \mathcal{F}_{i}$ can be crossed by few lines. We give optimal values for the number of piercing points and crossing lines needed when $d=2$ and also consider the problem restricted to special families of convex sets.
	\end{abstract}
	
	
	\maketitle
	
	\section{Introduction}
	In this paper we study problems with hypothesis similar to that of the Colorful Helly theorem, we are interested in finding transversals to the families involved.
	
	Let $\F$ be a family of sets in $\R^{d}$, a set $T \subset \R^{d}$ is a \textit{transversal} to the family $\F$ if every set $C \in \F$ intersects the set $T$. Additionally, if $T$ is a $k$-flat ($k$-dimensional affine subspace of $\R^{d}$) and is transversal to the family $\F$, we say that $T$ is a $k$-flat transversal. For example, a \textit{line transversal} is a line that intersects every member of $\F$ and a \textit{hyperplane transversal} is a hyperplane that intersects every member of $\F$.
	
	We say that the family $\F$ can be \textit{pierced} by $n$ points if there exist $n$ points $p_{1}, \dots, p_{n} \in \R^{d}$ such that every set $C \in \F$ intersects at least one of the points $p_{1}, \dots, p_{n}$; in other words, $P = \{ p_{1}, \dots, p_{n} \}$ is a transversal to the family $\F$. Similarly, we say that the family $\F$ can be \textit{crossed} by $n$ lines if there exist $n$ lines $l_{1}, \dots, l_{n} \in \R^{d}$ such that every set $C \in \F$ intersects at least one of the lines $l_{1}, \dots, l_{n}$; in other words, $l_{1} \cup \dots \cup l_{n}$ is a transversal to the family $\F$.
	
	We begin by giving a short list of theorems involving transversals including Helly-type theorems and Colorful theorems before presenting our main results in Section \ref{sec:main}.
	
	\subsection{Helly-type theorems}
	
	Helly's theorem \cite{Helly1923} is probably one of the most famous theorems in discrete and convex geometry. It states that if a finite family of convex sets in $\R^{d}$ satisfies that every $d+1$ or fewer of them have non-empty intersection, then the whole family has non-empty intersection. The number $ d+1 $ cannot be improved in Helly's theorem. However, it is possible to strengthen the hypothesis to obtain stronger results. For instance, Grünbaum proved the following theorem for families of homothetic copies of a convex set.
	
	\begin{theorem}[\cite{Grunbaum1959}]\label{thm:Grunbaum}
		For any integer $d \geq 1$ there exists an integer $c = c(d)$ such that the following holds. If $\F$ is a finite family of homothetic copies of a convex set in $\R^{d}$ and any two members of $\F$ have non-empty intersection, then $\F$ can be pierced by $c$ points.
	\end{theorem}
	
	%The case of circles in the plane was solved by Danzer.
	
	In \cite{Danzer1986} was solved the case of circles in the plane, in \cite{Dumitrescu2011} was solved the case of triangles in the plane and in \cite{Kim2006} was given a bound for the general planar case. In \cite{Kim2006}, \cite{Naszodi2010} and \cite{Dumitrescu2011} gave bounds for the general case.
	
	\begin{theorem}[\cite{Danzer1986}]\label{thm:Danzer}
		Let $\F$ be a finite family of circles in $\R^{2}$ such that the intersection of every $2$ sets in $\F$ is non-empty. Then $\F$ can be pierced by $4$ points. Furthermore, there are examples where $3$ points are not enough.
	\end{theorem}
	
	\begin{theorem}[\cite{Dumitrescu2011}]\label{thm:Dumitrescu}
		Let $\F$ be a finite family of homothetic copies of a triangle in $\R^{2}$ such that the intersection of every $2$ sets in $\F$ is non-empty. Then $\F$ can be pierced by $3$ points. Furthermore, there are examples where $2$ points are not enough.
	\end{theorem}
	
	\begin{theorem}[\cite{Kim2006}]\label{thm:Kim}
		Let $\F$ be a finite family of homothetic copies of a convex set in $\R^{2}$ such that the intersection of every $2$ sets in $\F$ is non-empty. Then $\F$ can be pierced by $16$ points.
	\end{theorem}	
	
	Karasev proved the following theorem for families of translates of a compact convex set in the plane.
	
	\begin{theorem}[\cite{Karasev2000}]\label{thm:Karasev}
		Let $K$ be a compact convex set in $\R^{2}$. Let $\F$ be a finite family of translates of $K$ such that the intersection of every $2$ sets in $\F$ is non-empty. Then $\F$ can be pierced by $3$ points. Furthermore, there are examples where $2$ points are not enough.
	\end{theorem}
	
	In fact, Karasev \cite{Karasev2008} also proved similar results in higher dimensions.
	We include another similar result concerning only axis-parallel boxes where the conclusion is strengthened.
	
	\begin{theorem}[\cite{Hadwiger1966, Fon1993}]\label{thm:rectangles}
		Let $\F$ be a finite family of axis-parallel boxes in $\R^{d}$ such that the intersection of every $2$ sets in $\F$ is non-empty. Then $\bigcap\F \neq \emptyset$.
	\end{theorem}
	
	\subsection{Colorful theorems}
	
	In 1973 Lovász proved a generalization of Helly's theorem called the Colorful Helly theorem.
	\begin{theorem}[\cite{Barany1982}]
		If we have $d+1$ finite families $\F_{1}, \dots, \F_{d+1}$ of convex sets in $\R^{d}$ such that for every choice of sets $C_{1} \in \F_{1}, \dots, C_{d+1} \in \F_{d+1}$, the intersection $\bigcap_{i=1}^{d+1} {C_{i}}$ is non-empty, then there exists $i \in \{ 1, \dots, d+1 \}$ such that the family $\F_{i}$ has non-empty intersection.
	\end{theorem}
	
	According to \cite{Jeronimo2015}, Dol’nikov wondered whether there is a colorful version of Theorem \ref{thm:Karasev}. Then Dol’nikov proposed the following conjecture.
	
	\begin{conjecture}[\cite{Jeronimo2015}]\label{conj:Dolnikov}
		Let $K$ be a compact convex set in $\R^{2}$. Let $\F_{1}, \F_{2}, \F_{3}$ be finite families of translates of $K$. Suppose that $A \cap B \neq \emptyset$ for every $A \in \F_{i}$ and $B \in \F_{j}$ with $i \neq j$. Then there exists $j \in \{ 1,2,3 \}$ such that $\F_{j}$ can be pierced by $3$ points.
	\end{conjecture}
	
	Jerónimo-Castro, Magazinov and Soberón \cite{Jeronimo2015} proved Conjecture \ref{conj:Dolnikov} when $K$ is centrally symmetric or a triangle. In fact, in the case where $K$ is a circle they proved a stronger statement: if $n \geq 2$ and $\F_{1}, \dots, \F_{n}$ are finite families of circles with the same radius in the plane such that $A \cap B \neq \emptyset$ for every $A \in \F_{i}$ and $B \in \F_{j}$ (with $i \neq j$), then there exists $j \in \{ 1,2, \dots, n \}$ such that $\bigcup_{i \neq j} \F_{i}$ can be pierced by $3$ points. They also made the following conjecture.
	
	\begin{conjecture}[\cite{Jeronimo2015}]\label{conj:DolnikovStrong}
		Let $K$ be a compact convex set in $\R^{2}$. Let $\F_{1}, \dots, \F_{n}$ be finite families of translates of $K$, with $n \geq 2$. Suppose that $A \cap B \neq \emptyset$ for every $A \in \F_{i}$ and $B \in \F_{j}$ with $i \neq j$. Then there exists $j \in \{ 1,2, \dots, n \}$ such that $\bigcup_{i \neq j} \F_{i}$ can be pierced by $3$ points.
	\end{conjecture}
	
	In 2020, Martínez-Sandoval, Roldán-Pensado and Rubin sought stronger conclusions for the Colorful Helly theorem \cite{Martinez2020}. In some sense they were looking for a \emph{Very Colorful Helly theorem}, similar to how the Colorful Carathéodory theorem has a Very Colorful version \cite{Holmsen2008,Arocha2009}.
	
	\begin{theorem}[\cite{Martinez2020}]\label{thm:VCH}
		For every $d \geq 2$ there exist numbers $f(d)$ and $g(d)$ with the following property: Let $\F_{1}, \dots, \F_{d}$ be finite families of convex sets in $\R^{d}$ such that for every choice of sets $C_{1} \in \F_{1}, \dots, C_{d} \in \F_{d}$, the intersection $ \bigcap_{i=1}^{d} {C_{i}}$ is non-empty. Then one of the following statements holds:
		\begin{enumerate}[\qquad 1.]
			\item there is a family $\F_{j}$, for $j \in \{ 1, \dots, d \}$, that can be pierced by $f(d)$ points, or
			\item the family $\bigcup_{i=1}^{d} \F_{i}$ can be crossed by $g(d)$ lines.
		\end{enumerate}
	\end{theorem}
	
	In particular, they showed that the $2$-dimensional case of this theorem holds with $f(2) = 1$ and $g(2) = 4$. One may ask for the optimal pairs $(f(d),g(d))$ in this theorem, for example, does Theorem \ref{thm:VCH} hold with $f(d) = 1$ and large enough $g(d)$?
	
	\section{Main results}\label{sec:main}
	
	We prove Conjecture \ref{conj:Dolnikov} when $K$ is of constant width and when $K$ is close to a circle in Banach-Mazur distance. For an arbitrary convex body, Conjecture \ref{conj:Dolnikov} holds with $8$ piercing points instead of $3$. Additionally, we show that Conjecture \ref{conj:DolnikovStrong} holds with $9$ piercing points and that, when $K$ has constant width, Conjecture \ref{conj:DolnikovStrong} holds with $4$ piercing points. Furthermore, we prove similar results in arbitrary dimension.
	
	It should be noted that properties we are studying are invariant under linear transformations, so if any of the following theorems is true for a convex body $K$ then it is also true for the image of $K$ under a linear isomorphism.
	
	\begin{theorem}\label{thm:DolnikovCW}
		Let $K$ be a convex body in $\R^2$. Let $\F_1$ and $\F_2$ be finite families of translates of $K$ such that $A \cap B \neq \emptyset$ for every $A \in \F_1$ and $B \in \F_2$. If $K$ is of constant width or if $K$ has Banach-Mazur distance at most $1.1178$ to the disk, then either $\F_1$ or $\F_2$ can be pierced by $3$ points.
	\end{theorem}
	
	The proof of this theorem, included in Section \ref{sec:DolnikovCW}, can be extended to give an upper bound for the number of piercing points needed for any $K$. For each $K$ a certain pentagon is constructed, the bound is dependent on the maximum number of translates of $K$ needed to cover every rotation of this pentagon.
	As corollary of this proof, using the fact that any convex body is at Banach-Mazur distance at most $2$ from the unit disk, we obtain the following.
	
	\begin{coro}
		Let $K$ be a compact convex set in $\R^{2}$. Let $\F_{1}$ and $\F_{2}$ be finite families of translates of $K$. Suppose that $A \cap B \neq \emptyset$ for every $A \in \F_{i}$ and $B \in \F_{j}$ with $i \neq j$. Then there exists $j \in \{ 1,2 \}$ such that $\F_{j}$ can be pierced by $8$ points.
	\end{coro}
	
	In particular, this implies a weaker form of Dol’nikov's conjecture (Conjecture \ref{conj:Dolnikov}) where $8$ points are needed to pierce one of the three families.

	\begin{theorem}\label{thm:DolnikovNew}
		Let $K$ be a convex body in $\R^d$, then there exists a number $f_K(d)$ with the following property. If $\F_{1}, \dots, \F_{n}$ are finite families of translates of $K$ with $n \geq 2$ such that $A \cap B \neq \emptyset$ for every $A \in \F_{i}$ and $B \in \F_{j}$ with $i \neq j$, then there exists $j \in \{ 1,2, \dots, n \}$ such that $\bigcup_{i \neq j} \F_{i}$ can be pierced by $f_K(d)$ points.
		Moreover, in the planar case we have the following:
		\begin{enumerate}[\qquad (a)]
			\item we may choose $f_K(2)=9$ for any $K$ and
			\item we may choose $f_K(2)=4$ if $K$ has constant width.
		\end{enumerate}
	\end{theorem}
	
	The proof of this theorem is included in Section \ref{sec:DolnikovNew} below.
	
	Using the KKM theorem \cite{Knaster1929}, we show that Theorem \ref{thm:VCH} holds with $f(2) = 1$ and $g(2) = 2$. In \cite{Martinez2020} it is shown that if $g(d)$ is a number for which Theorem \ref{thm:VCH} holds then $g(d) \geq \bigl\lceil \frac{d+1}{2} \bigr\rceil$. Therefore the values $f(2) = 1$ and $g(2) = 2$ cannot be improved. We actually prove a more general result in $\R^2$.
	
	\begin{theorem}\label{thm:PlaneVCH}
		Let $\F_{1}, \dots, \F_{n}$ be finite families of convex sets in $\R^{2}$ with $n \geq 2$. Suppose that $A \cap B \neq \emptyset$ for every $A \in \F_{i}$ and $B \in \F_{j}$ with $i \neq j$. Then one of the following statements holds:
		\begin{enumerate}[\qquad 1.]
			\item there exists $j \in \{ 1,2, \dots, n \}$ such that $\bigcup_{i \neq j} \F_{i}$ can be pierced by $1$ point, or
			\item the family $\bigcup_{i=1}^{n} \F_{i}$ can be crossed by $2$ lines.
		\end{enumerate}
	\end{theorem}
	
	The proof of this result is included in Section \ref{sec:PlaneVCH} and first appeared in the master's thesis of the first author \cite{Gomez-Navarro2022} in 2022. Even though the $2$-dimensional case of Theorem \ref{thm:VCH} does not hold with $g(2) = 1$, it does for some special families of convex sets.
	
	\begin{theorem}\label{thm:SpecialVCH}
		Let $\K$ a subfamily of the family of convex sets in $\R^d$, then there exist numbers $f_\K(d)$ and $g_\K(d)$ with the following property. If $\F_{1}, \dots, \F_{n}$ are finite families of sets in $\K$ with $n\ge 2$ such that $A \cap B \neq \emptyset$ for every $A \in \F_{i}$ and $B \in \F_{j}$ with $i \neq j$, then one of the following statements holds:
		\begin{enumerate}[\qquad 1.]
			\item there exists $j \in \{ 1,2, \dots, n \}$ such that $\bigcup_{i \neq j} \F_{i}$ can be pierced by $f_\K(d)$ points, or
			\item the family $\bigcup_{i=1}^{n} \F_{i}$ can be crossed by $g_\K(d)$ hyperplanes.
		\end{enumerate}
		Here we may choose $f_\K(d)=1$ and $g_\K(d)=d$.
		Moreover, for $d=2$ and specific families $\K$ we have the following:
		\begin{enumerate}[\qquad (a)]
			\item if $\K$ is the family of translates of a given compact convex body we may choose $f_\K(2)=3$ and $g_\K(2)=1$,
			\item if $\K$ is the family of homothetic copies of a given compact convex body we may choose $f_\K(2)=16$ and $g_\K(2)=1$,
			\item if $\K$ is the family of homothetic copies of a triangle we may choose $f_\K(2)=3$ and $g_\K(2)=1$, and
			\item if $\K$ is the family of circles we may choose $f_\K(2)=4$ and $g_\K(2)=1$.
			%\item if $\K$ is the family of rectangles with sides parallel to the coordinate axes we may choose $f_\K(2)=1$ and $g_\K(2)=1$.
		\end{enumerate}
		For general $d$ we have the following:
		\begin{enumerate}[\qquad (a)]\setcounter{enumi}{4}
			\item if $\K$ is the family of axis-parallel boxes we may choose $f_\K(d)=1$ and $g_\K(d)=1$, and
			\item if $\K$ is the family of homothetic copies of a given compact convex body we may choose $g_\K(d)=1$.
		\end{enumerate}
	\end{theorem}
	
	Note that in this theorem, statement 2 is about crossing hyperplanes while in Theorem \ref{thm:VCH} it is about crossing lines. The proof is included in Section \ref{sec:SpecialVCH}.
	
	\section{Proof of Theorem \ref{thm:PlaneVCH}}\label{sec:PlaneVCH}
	
	We rely on the KKM theorem \cite{Knaster1929} in order to prove Theorem \ref{thm:PlaneVCH}. Let
	\begin{align*}
		\Delta^{n} = \left\{ (x_{1}, \dots, x_{n+1}) \in \mathbb{R}^{n+1} \ : \ x_{i} \geq 0, \ \sum_{i=1}^{n+1} x_{i} = 1 \right\}
	\end{align*}
	denote the $n$-dimensional simplex in $\mathbb{R}^{n+1}$ that is the convex hull of $\{ e_{1}, \dots, e_{n+1} \}$, the standard orthonormal basis of $\mathbb{R}^{n+1}$.
	
	\begin{theorem}[KKM theorem \cite{Knaster1929}]\label{KKM theorem}
		Let $\{ O_{1}, O_{2}, \dots, O_{n+1} \}$ be an open (or closed) cover of $\Delta^{n}$ such that $e_{i} \in O_{i}$ for each $i \in \{ 1,2, \dots, n+1 \}$, and $\conv \{ e_{i} : i \in I \} \subset \bigcup_{i \in I} O_{i}$ for each $I \subset \{ 1,2, \dots, n+1 \}$.
		Then $\bigcap_{i=1}^{n+1} O_{i} \neq \emptyset$.
	\end{theorem}
	
	In order to prove Theorem \ref{thm:PlaneVCH} we follow ideas similar to the ones used in \cite{Mcginnis2021}.
	
	\begin{proof}[Proof of Theorem \ref{thm:PlaneVCH}]
		We can assume, without loss of generality, that the sets in the finite families are compact (see \cite[Chapter 1]{Gomez-Navarro2022}). Hence, we may scale the plane such that every set in $\bigcup_{i=1}^{n} \F_{i}$ is contained in the unit disk. Let $f(t)$ be a parametrization of the unit circle defined by
		\begin{align*}
			f(t) = (\cos(2 \pi t), \ \sin(2 \pi t)).
		\end{align*}
		To each point $x=(x_{1}, \dots, x_{4}) \in \Delta^{3}$ we associate $4$ points on the unit circle given by
		\begin{align*}
			f_{i}(x) = f \left( \sum_{j=1}^{i} {x_{j}} \right),
		\end{align*}
		for $1 \leq i \leq 4$. Let $l_{1}(x) = l_{3}(x) = [f_{1}(x), f_{3}(x)]$ and $l_{2}(x) = l_{4}(x) = [f_{2}(x), f_{4}(x)]$. For $i = 1, \dots, 4$ let $R_{x}^{i}$ be the interior of the region bounded by $l_{i-1}(x), l_{i}(x)$ and the arc on the unit circle connecting $f_{i-1}(x)$ and $f_{i}(x)$, where $i-1$ is taken modulo $4$ (see Figure \ref{figure_proof_KKM}).
		
		\begin{figure}
			\includegraphics{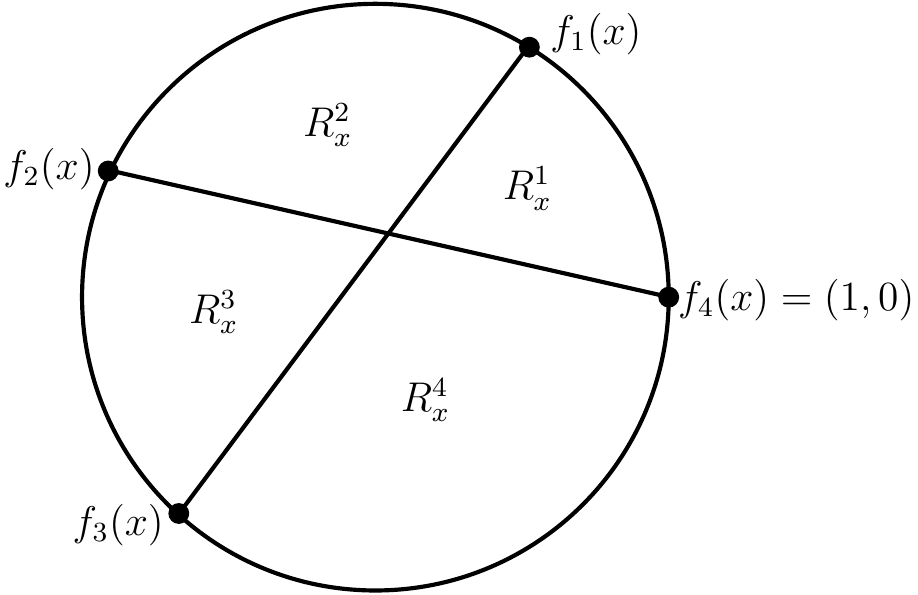}
			\caption{Illustration for the proof of Theorem \ref{thm:PlaneVCH}.}
			\label{figure_proof_KKM}
		\end{figure}
		
		Notice that $f_{4}(x) = (1,0)$ for each $x \in \Delta^{3}$. Also, the points $f_{1}(x), f_{2}(x), f_{3}(x), f_{4}(x)$ are always in counter-clockwise order so the lines $l_{1}(x) = [f_{1}(x), f_{3}(x)]$ and $l_{2}(x) = [f_{2}(x), f_{4}(x)]$ intersect.
		
		If there is some $x \in \Delta^{3}$ for which $l_{1}(x) \cup l_{2}(x)$ is a transversal to the family $\bigcup_{i=1}^{n} \F_{i}$, the second statement of the theorem holds and we are done. Otherwise, since the sets in $\bigcup_{i=1}^{n} \F_{i}$ are convex, we may assume that for every $x \in \Delta^{3}$ there is a set $C \in \bigcup_{i=1}^{n} \F_{i}$ contained in one of the four open regions $R_{x}^{i}$. For $i = 1, \dots, 4$, let $O_{i}$ be the set of points $x \in \Delta^{3}$ such that $R_{x}^{i}$ contains a set $C \in \bigcup_{i=1}^{n} \F_{i}$. Since the sets $C \in \bigcup_{i=1}^{n} \F_{i}$ are compact, $O_{i}$ is open. Notice that $\Delta^{3} = \bigcup_{i=1}^{4} {O_{i}}$, since for every $x \in \Delta^{3}$ there is a set $C \in \bigcup_{i=1}^{n} \F_{i}$ contained in one of the four open regions $R_{x}^{i}$. Observe that if $x \in \conv \{ e_{i} : i \in I \}$ for some $I \subset \{ 1, \dots, 4 \}$, then $R_{x}^{j} = \emptyset$ for $j \notin I$, so $x \in \bigcup_{i \in I} {O_{i}}$.
		Therefore $\{ O_{1}, \dots, O_{4} \}$ is an open cover that satisfies the hypothesis of the KKM theorem (Theorem \ref{KKM theorem}). Consequently there is a point $y = (y_{1}, \dots, y_{4}) \in \Delta^{3}$ such that $y \in \bigcap_{i=1}^{4} {O_{i}}$. In other words, each one of the open regions $R_{y}^{i}$ contains a set $C_{i} \in \bigcup_{i=1}^{n} \F_{i}$.
		
		\begin{figure}
			\includegraphics{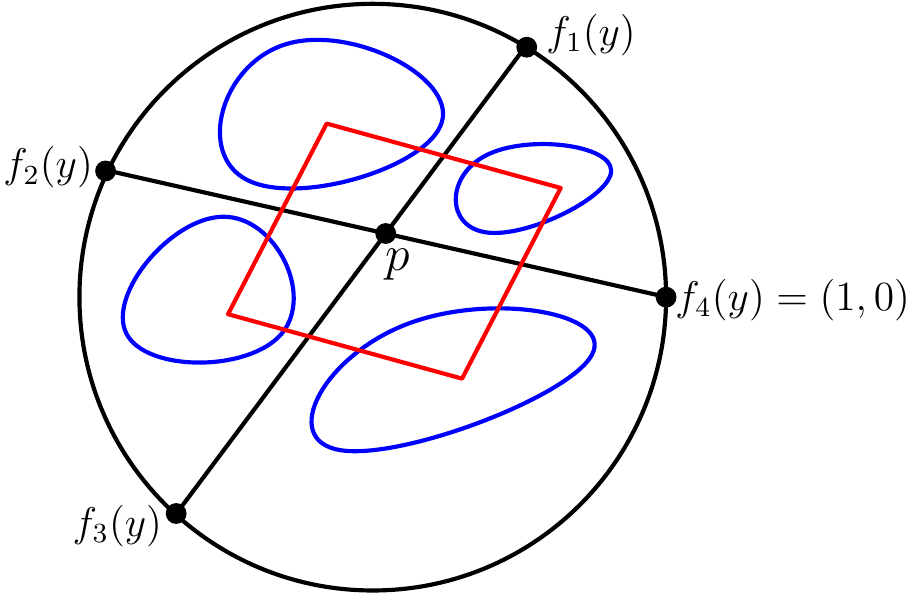}
			\caption{If the blue sets are in the family $\mathcal{F}_{j}$, then every set in the family $\bigcup_{i \neq j} \mathcal{F}_{i}$ contains the point $p$.}
			\label{figure_1}
		\end{figure}
		
		Since the sets $C_{1}, \dots, C_{4}$ are pairwise disjoint, the sets $C_{1}, \dots, C_{4}$ are in the same family $\F_{j}$, for some $j \in \{ 1, \dots, n \}$. Let $p$ be the point intersection of the lines $l_{1}(y) = l_{3}(y)$ and $l_{2}(y) = l_{4}(y)$. Take any set $B \in \bigcup_{i \neq j} \F_{i}$. Since $B \cap C_{i} \neq \emptyset$ for every $i \in \{ 1, \dots, 4 \}$ and $B$ is convex, then $B$ intersects the line segments $[p, f_{4}(y)]$ and $[p, f_{2}(y)]$. Therefore, $p \in B$ and the first statement of the theorem holds (see Figure \ref{figure_1}).
	\end{proof}
	
	\section{Proof of Theorem \ref{thm:SpecialVCH}}\label{sec:SpecialVCH}
	
	Here we prove that for some special families of convex sets, the $2$-dimensional case of Theorem \ref{thm:VCH} holds with $g(2) = 1$. We also show similar results in higher dimensions and as a corollary of this we obtain a second proof of Theorem \ref{thm:PlaneVCH}.
	
	\begin{proof}[Proof of Theorem \ref{thm:SpecialVCH}]
		%To show that the numbers $f_\K(d)$ and $g_\K(d)$ exist we use Lemma 2.1 from \cite{Martinez2020}. This Lemma states the following: If $\mathcal A$ and $\mathcal B$ are finite families of convex sets in $\mathbb{R}^{d}$ such that $A \cap B \neq \emptyset$ for every $A \in \mathcal{A}$ and $B \in \mathcal{B}$, then either $\bigcap \mathcal{A} \neq \emptyset$ or $\mathcal{B}$ can be crossed by $d$ hyperplanes. Even though there are only two families involved in this lemma, it can be used to prove the following.
		
		To show that the numbers $f_\K(d)$ and $g_\K(d)$ exist we use the following two auxiliary results.
		
		%\begin{lemma}\label{lem:MSRPR}
			%Let $\mathcal{F}_{1}, \dots, \mathcal{F}_{n}$ be finite families of convex sets in $\mathbb{R}^{d}$ with $n \geq 2$ such that $A \cap B \neq \emptyset$ for every $A \in \mathcal{F}_{i}$ and $B \in \mathcal{F}_{j}$ with $i \neq j$. Then one of the following statements holds:
			%\begin{enumerate}[\qquad 1.]
				%\item there exists $j \in \{ 1,2, \dots, n \}$ such that $\bigcup_{i \neq j} \mathcal{F}_{i}$ can be pierced by $1$ point, or
				%\item the family $\bigcup_{i=1}^{n} \mathcal{F}_{i}$ can be crossed by $dn$ hyperplanes.
			%\end{enumerate}
		%\end{lemma}
		%\begin{proof}
			%Assume that for every $j \in \{ 1,2, \dots, n \}$ the family $\bigcup_{i \neq j} \mathcal{F}_{i}$ cannot be pierced by $1$ point. For every $j \in \{ 1,2, \dots, n \}$ we take $\mathcal{A} = \bigcup_{i \neq j} \mathcal{F}_{i}$ and $\mathcal{B} = \mathcal{F}_{j}$. By Lemma 2.1 from \cite{Martinez2020}, the family $\mathcal{B} = \mathcal{F}_{j}$ can be crossed by $d$ hyperplanes. Thus, $\bigcup_{i=1}^{n} \mathcal{F}_{i}$ can be crossed by $dn$ hyperplanes.
		%\end{proof}
		
		\begin{lemma}[\cite{Martinez2020}]\label{lem:MSRPR}
		If $\mathcal A$ and $\mathcal B$ are finite families of convex sets in $\mathbb{R}^{d}$ such that $A \cap B \neq \emptyset$ for every $A \in \mathcal{A}$ and $B \in \mathcal{B}$, then either $\bigcap \mathcal{A} \neq \emptyset$ or $\mathcal{B}$ can be crossed by $d$ hyperplanes.
		\end{lemma}
		
		%Notice that in this lemma from \cite{Martinez2020} there are only two families involved and the second conclusion concerns only one family, while Theorem \ref{thm:SpecialVCH} not only involves more than two families but also the second conclusion concerns all the families. Of course, 
		As was noted in \cite{Martinez2020}, by applying Lemma \ref{lem:MSRPR} twice we obtain that either $\mathcal A$ or $\mathcal B$ can be pierced by $1$ point or $\mathcal A \cup \mathcal B$ can be crossed by $2d$ hyperplanes. However, these numbers are not optimal. We use the following auxiliary result (which is also useful later) to prove that Theorem \ref{thm:SpecialVCH} holds with $f_{\K}(d) = 1$ and $g_{\K}(d) = d$.
		
		%This lemma shows that we may always choose $f_{\K}(d) = 1$ and $g_{\K}(d) = nd$. To prove the rest of the statements, we use the following auxiliary result which is also useful later.
		
		\begin{lemma}\label{lem:specialVCH}
			Let $\mathcal{F}_{1}, \dots, \mathcal{F}_{n}$ be finite families of convex sets in $\mathbb{R}^{d}$ with $n \geq 2$ such that $A \cap B \neq \emptyset$ for every $A \in \mathcal{F}_{i}$ and $B \in \mathcal{F}_{j}$ with $i \neq j$. Then one of the following statements holds:
			\begin{enumerate}
				\item[1.] there exists $j \in \{ 1,2, \dots, n \}$ such that every two sets in $\bigcup_{i \neq j} \mathcal{F}_{i}$ intersect, or
				\item[2.] the family $\bigcup_{i=1}^{n} \mathcal{F}_{i}$ has a hyperplane transversal.
			\end{enumerate}
		\end{lemma}
		
		\begin{proof}
			As usual, we may assume without loss of generality that the sets in $\bigcup_{i=1}^{n} \mathcal{F}_{i}$ are compact.
			
			For every direction $u \in \mathbb{S}^{d-1}$, let $\ell_{u}$ be the line through the origin with direction $u$. By projecting the sets in $\bigcup_{i=1}^{n} \mathcal{F}_{i}$ to the line $\ell_{u}$ we obtain a finite family of intervals in the line $\ell_{u}$.
			If for some $u \in \mathbb{S}^{d-1}$ the intervals in the line $\ell_{u}$ have a common point $p$ then the hyperplane through $p$ orthogonal to $u$ is transversal to the family $\bigcup_{i=1}^{n} \mathcal{F}_{i}$, and we are done.
			Otherwise, for every $u \in \mathbb{S}^{d-1}$, there are two disjoint intervals in the line $\ell_{u}$. Hence, for every $u \in \mathbb{S}^{d-1}$, there are two sets $A_{u}, B_{u}$ in the family $\bigcup_{i=1}^{n} \mathcal{F}_{i}$ that are separated by a hyperplane orthogonal to $u$. Since $A_{u} \cap B_{u} = \emptyset$, then $A_{u}$ and $B_{u}$ must be in the same family $\mathcal{F}_{i}$, for some $i \in \{ 1, \dots, n \}$.
			
			We color the sphere $\mathbb{S}^{d-1}$ as follows. If there are two sets $A_{u}, B_{u} \in \mathcal{F}_{i}$ that are separated by a hyperplane orthogonal to $u$, we color $u \in \mathbb{S}^{d-1}$ with color $i$. Let $O_{i}$ be the subset of $\mathbb{S}^{d-1}$ with color $i$. Since the convex sets are compact, the sets $O_{i}$ are open. Notice that the sets $O_{1}, \dots, O_{n}$ cover $\mathbb{S}^{d-1}$. We consider two cases.
			
			First suppose that at least two sets from $O_{1}, \dots, O_{n}$ are non-empty. Since the sets $O_{1}, \dots, O_{n}$ are open, there are two indices $i,j \in \{ 1, \dots, n \}$ with $i \neq j$ such that $O_{i}$ and $O_{j}$ intersect. Let $u\in O_i\cap O_j$, then there are hyperplanes $H$ and $G$ orthogonal to $u$, and sets $A,B\in \mathcal{F}_{i}$ and $C,D\in\mathcal{F}_{j}$ such that $H$ strictly separates $A$ from $B$, and $G$ strictly separates $C$ from $D$. This implies that one of the sets in $\{A,B\}$ is strictly separated from one of the sets in $\{C,D\}$ (see Figure \ref{figure_proof_Lemma}) which contradicts the hypothesis of the lemma.
			
			\begin{figure}
				\includegraphics{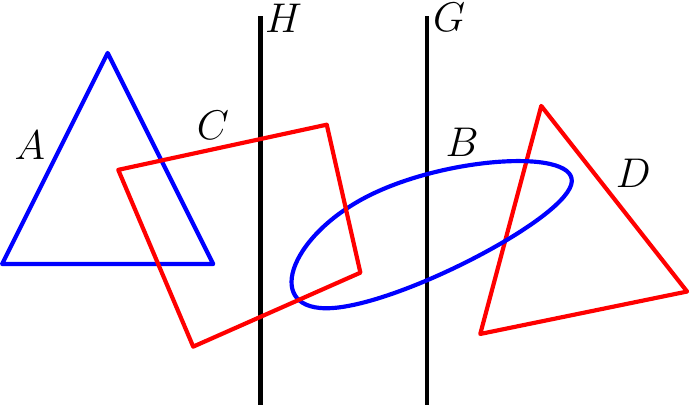}
				\caption{In this example $A \cap D = \emptyset$.}
				\label{figure_proof_Lemma}
			\end{figure}
			
			In the second case there exists $j \in \{ 1, \dots, n \}$ such that $\mathbb{S}^{d-1} = O_{j}$, so for every $u \in \mathbb{S}^{d-1}$ there is a hyperplane $H_{u}$ orthogonal to $u$ and sets $A_{u}, B_{u} \in \mathcal{F}_{j}$ such that $H_u$ strictly separates $A_{u}$ from $B_{u}$. Let $C \in \bigcup_{i \neq j} \mathcal{F}_{i}$, since $C$ intersects $A_{u}$ and $B_{u}$ it must also intersect $H_u$. This shows that the family $\bigcup_{i \neq j} \mathcal{F}_{i}$ has a hyperplane transversal in every direction.
			
			If there are two sets $C, D \in \bigcup_{i \neq j} \mathcal{F}_{i}$ such that $C \cap D = \emptyset$, then $C$ and $D$ can be separated by a hyperplane $H$. Hence $C$ and $D$ do not have a hyperplane transversal parallel to $H$, a contradiction. Therefore every two sets in $\bigcup_{i \neq j} \mathcal{F}_{i}$ intersect.
		\end{proof}
		
		We prove that Theorem \ref{thm:SpecialVCH} holds with $f_{\K}(d) = 1$ and $g_{\K}(d) = d$. If $\bigcup_{i=1}^{n} \mathcal{F}_{i}$ has a hyperplane transversal, we are done. Otherwise, by Lemma \ref{lem:specialVCH}, there is $j \in \{ 1, \dots, n \}$ such that every two sets in $\bigcup_{i \neq j} \mathcal{F}_{i}$ intersect. Then the families $\mathcal{A} = \bigcup_{i \neq j} \mathcal{F}_{i}$ and $\mathcal{B} = \bigcup_{i=1}^{n} \mathcal{F}_{i}$ satisfy the hypothesis of Lemma \ref{lem:MSRPR}. Therefore, either $\mathcal{A} = \bigcup_{i \neq j} \mathcal{F}_{i}$ can be pierced by $1$ point or $\mathcal{B} = \bigcup_{i=1}^{n} \mathcal{F}_{i}$ can be crossed by $d$ hyperplanes. This shows that we may always choose $f_{\K}(d) = 1$ and $g_{\K}(d) = d$. Notice that as a corollary of this we obtain a second proof of Theorem \ref{thm:PlaneVCH}.
		
		We are now ready to prove the parts (a) to (f) of Theorem \ref{thm:SpecialVCH}. In each of these cases, if the family $\bigcup_{i=1}^{n} \mathcal{F}_{i}$ has a hyperplane transversal, we are done. Otherwise, by Lemma \ref{lem:specialVCH}, there exists $j \in \{ 1, \dots, n \}$ such that every two sets in $\bigcup_{i \neq j} \mathcal{F}_{i}$ intersect. If this last condition implies that $\bigcup_{i \neq j} \mathcal{F}_{i}$ can be pierced by $f_\K(d)$ points, we are done.
		For parts (a), (b), (c), (d), (e) and (f) we conclude by using Theorems \ref{thm:Karasev}, \ref{thm:Kim}, \ref{thm:Dumitrescu}, \ref{thm:Danzer}, \ref{thm:rectangles} and \ref{thm:Grunbaum}, respectively.
	\end{proof}

	\section{Proof of Theorem \ref{thm:DolnikovNew}}\label{sec:DolnikovNew}
	
	In \cite{Martinez2020} it is shown that if instead of using $d+1$ finite families in the Colorful Helly theorem we use only $d$, then it may be false that one of the families can be pierced by few points. In this section we prove that if the families are translates of a convex body, one of the families can indeed be pierced by few points.
	
	\begin{proof}[Proof of Theorem \ref{thm:DolnikovNew}]
		By part (f) of Theorem \ref{thm:SpecialVCH}, either there is $j \in \{ 1, \dots, n \}$ such that $\bigcup_{i \neq j} \mathcal{F}_{i}$ can be pierced by $f_K(d)$ points or the family $\bigcup_{i=1}^{n} \mathcal{F}_{i}$ has a hyperplane transversal. In the first case we are done so assume that $\bigcup_{i=1}^{n} \mathcal{F}_{i}$ has a hyperplane transversal in the direction $v_{1} \in \mathbb{S}^{d-1}$.
		
		For every $u$ there is either a hyperplane orthogonal to $u$ transversal to $\bigcup_{i=1}^{n} \mathcal{F}_{i}$ or there are two sets in some family $\mathcal{F}_{k}$ separated by a hyperplane $H_u$ orthogonal to $u$. In the second case $H_u$ is transversal to the family $\bigcup_{i \neq k} \mathcal{F}_{i}$. In any case, for every direction $u$, there is a hyperplane transversal to at least $n-1$ of the families $\mathcal{F}_{1}, \dots, \mathcal{F}_{n}$.
		
		We color $u \in \mathbb{S}^{d-1}$ with color $i$ if there is a hyperplane orthogonal to $u$ and transversal to the family $\mathcal{F}_{i}$. Then every point in $\mathbb{S}^{d-1}$ is colored with at least $n-1$ colors. Moreover, if a vector $v$ is colored with all $n$ colors, there is a hyperplane orthogonal to $v$ transversal to the family $\bigcup_{i=1}^{n} \mathcal{F}_{i}$. For instance, $v_{1}$ is colored with all $n$ colors.
		
		For every $i \in \{ 1, \dots, n \}$, let $F_{i}$ be the set of all the points in the sphere $\mathbb{S}^{d-1}$ with color $i$. Since the sets in the families are compact, then the sets $F_{i}$ are closed.
		
		\begin{lemma}\label{lem:colorful sphere}
			If the sphere $\mathbb{S}^{d-1}$ is covered by $n$ closed families $F_{1}, \dots, F_{n}$ such that every $u \in \mathbb{S}^{d-1}$ belongs to at least $n-1$ families and there exists $v_{1} \in \bigcap_{i=1}^{n} F_{i}$, then there is an orthonormal basis $\{ v_{1}, \dots, v_{d} \}$ of $\mathbb{R}^{d}$ and an index $j \in \{ 1, \dots, n \}$ such that $v_{1}, \dots, v_{d}\in\bigcap_{i \neq j} F_{i}$.
		\end{lemma}
		
		\begin{proof}[Proof of Lemma \ref{lem:colorful sphere}]
			We proceed by induction on $d$. If $d=2$, we extend $v_{1}$ to an orthonormal basis $\{ v_{1}, v_{2} \}$ of $\mathbb{R}^{2}$. By hypothesis, there exists $j \in \{ 1, \dots, n \}$ such that $v_{2} \in \bigcap_{i \neq j} F_{i}$. Thus, $v_{1}, v_{2}$ are both in $\bigcap_{i \neq j} F_{i}$.
			
			For $d>2$ we consider the $(d-2)$-dimensional sphere $S \subset \mathbb{S}^{d-1}$ orthogonal to $v_{1}$. There are two cases. First suppose that there exists $j \in \{ 1, \dots, n \}$ such that all the points in $S$ are in $\bigcap_{i \neq j} F_{i}$. Then we extend $v_{1}$ to an orthonormal basis $\{ v_{1}, \dots, v_{d} \}$ of $\mathbb{R}^{d}$. Since $\{ v_{2}, \dots, v_{d} \} \subset S$, then $v_{1}, \dots, v_{d}$ are in $\bigcap_{i \neq j} F_{i}$ and we are done.
			Otherwise, suppose that there are $j, k \in \{ 1, \dots, n \}$ with $j \neq k$ such that $S$ contains points from both $\bigcap_{i \neq j} F_{i}$ and $\bigcap_{i \neq k} F_{i}$. Since $S$ is connected and the sets $F_{i}$ are closed, there is a point $v_{2} \in S\cap \bigcap_{i=1}^{n} F_{i}$. By using the induction hypothesis on the $(d-1)$-dimensional sphere $S$, we can complete $\{v_1,v_2\}$ to an orthonormal basis $\{ v_{1}, \dots, v_{d} \}$ of $\mathbb{R}^{d}$ such that $v_{1}, \dots, v_{d}\in\bigcap_{i \neq j} F_{i}$ for some $j \in \{ 1, \dots, n \}$.
		\end{proof}
		
		This lemma gives us an orthonormal basis $\{ v_{1}, \dots, v_{d} \}$ of $\mathbb{R}^{d}$ and $j \in \{ 1, \dots, n \}$ such that the family $\bigcup_{i \neq j} \mathcal{F}_{i}$ has hyperplane transversals $H_{1}, \dots, H_{d}$ orthogonal to $v_{1}, \dots, v_{d}$, respectively.
		
		Without loss of generality we may assume that $K$ is compact and has non-empty interior. Then bye John's ellipse theorem \cite{John1948} there is an ellipse $E\subset K$ with center $c$ such that $E\subset K\subset c+d(E-c)$. After applying a linear transformation, we may assume that $E$ is a unit ball. Let $\mathcal{G}_{j}$ be the family of balls of radius $d$ that contain some set from $\bigcup_{i \neq j} \mathcal{F}_{i}$ and let $\mathcal{H}_{j}$ be the family of balls of radius $1$ that are contained in some set from $\bigcup_{i \neq j} \mathcal{F}_{i}$. Since every set in $\bigcup_{i \neq j} \mathcal{F}_{i}$ intersects the hyperplanes $H_{1}, \dots, H_{d}$, it follows that for every direction $v_{i}$, the centers of the balls in $\mathcal{G}_{j}$ are contained in the strip bounded by two hyperplanes parallel to $H_{i}$. The width of this strips is $2 d$. Therefore the centers of the balls in $\mathcal{G}_{j}$ are contained in a $d$-dimensional cube with side $2 d$.
		
		The $d$-dimensional cube with side $2d$ can be covered by $f_K(d)$ unit balls \cite{Rogers1963} (see also \cite{Verger2005}), where $f_K(d)$ depends only on the dimension. This implies that every ball in $\mathcal{H}_{j}$ intersects at least one of the centers of these unit balls. Therefore, the family $\bigcup_{i \neq j} \mathcal{F}_{i}$ can be pierced by $f_K(d)$ points.
		
		To prove part (a) of Theorem \ref{thm:DolnikovNew} this bound can be made precise. We require to cover a square of side $4$ with unit circles. This can easily be done with $9$ circles as shown on the left side of Figure \ref{fig:circles-cover}. This implies that the family $\bigcup_{i \neq j} \mathcal{F}_{i}$ can be pierced by $9$ points.
		
		\begin{figure}
			\includegraphics{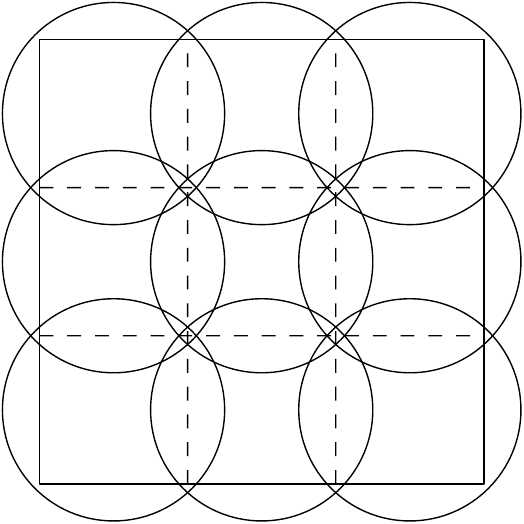}\qquad
			\includegraphics{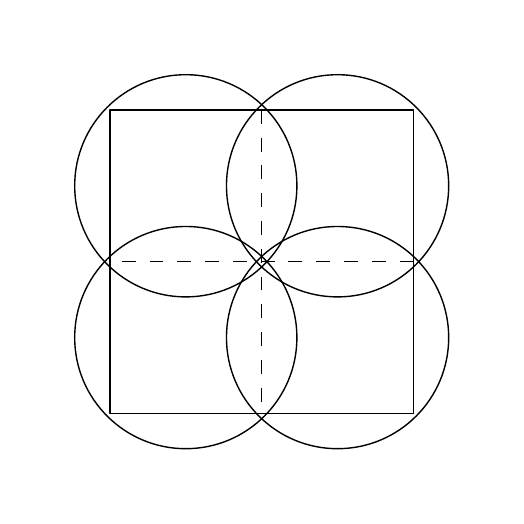}
			\caption{The square of side $4$ can be covered with $9$ unit circles, a square of side $1+\sqrt{3}$ can be covered with $4$ unit circles.}
			\label{fig:circles-cover}
		\end{figure}
		
		For part (b), if $K\subset\R^2$ is of constant width then, after applying a suitable linear transformation, we may assume that $K$ contains a unit circle and is contained in a circle of radius $\frac 12(1+\sqrt{3})$ (see e.g. \cite[Theorems 3.4.1 and 3.4.2]{Martini2019}). We then continue as before and we are left with a square of side $1+\sqrt{3}$ which we need to cover with unit circles. This may be accomplished with $4$ circles as shown on the right side of Figure \ref{fig:circles-cover}. This implies that the family $\bigcup_{i \neq j} \mathcal{F}_{i}$ can be pierced by $4$ points.
	\end{proof}
	
	To prove Theorem \ref{thm:DolnikovNew} we used Theorem \ref{thm:SpecialVCH}, and to prove Theorem \ref{thm:SpecialVCH} we used the monochromatic theorems. However, the monochromatic theorems are not necessary to prove Theorem \ref{thm:DolnikovNew}. Indeed, we recall that in the proof of Theorem \ref{thm:SpecialVCH} we applied Lemma \ref{lem:specialVCH} and we used the monochromatic theorems in the case when there is $j \in \{ 1, \dots, n \}$ such that every two sets in $\bigcup_{i \neq j} \mathcal{F}_{i}$ intersect. In the proof of Theorem \ref{thm:DolnikovNew} we can avoid the monochromatic theorems since the fact that every two sets in $\bigcup_{i \neq j} \mathcal{F}_{i}$ intersect is equivalent to that the family $\bigcup_{i \neq j} \mathcal{F}_{i}$ has hyperplane transversals in every direction (in the plane we have line transversals), and we can conclude in the same way as in the other case proved in Theorem \ref{thm:DolnikovNew}. Thus, to prove the bounds given in Theorem \ref{thm:DolnikovNew} we only need our Lemma \ref{lem:specialVCH}.
	
	\section{Proof of Theorem \ref{thm:DolnikovCW}}\label{sec:DolnikovCW}
	
	In this section we prove that Dol’nikov's conjecture (Conjecture \ref{conj:Dolnikov}) holds if $K$ is of constant width or if $K$ has Banach-Mazur distance at most $1.1178$ to the disk. Our results are slightly stronger than what Dol’nikov's conjecture would imply since we use $2$ only colors instead of $3$.
	
	\begin{proof}[Proof of Theorem \ref{thm:DolnikovCW}]
		We proceed in a way similar to the proof of Theorem \ref{thm:DolnikovNew}.
		If $\mathcal{F}_{1}$ or $\mathcal{F}_{2}$ can be pierced by $3$ points, we are done. Otherwise, by part (a) of Theorem \ref{thm:SpecialVCH}, $\mathcal{F}_{1} \cup \mathcal{F}_{2}$ has a line transversal $l$. Let $u_{1} \in \mathbb{S}^{1}$ be a direction orthogonal to $l$.
		
		A consequence of the $1$-dimensional colorful Helly theorem is that in every direction there is a line transversal to either $\mathcal{F}_{1}$ or $\mathcal{F}_{2}$. We color the sphere $\mathbb{S}^{1}$ so that $u\in\mathbb{S}^{1}$ has color $i$ whenever there is a line transversal to the family $\mathcal{F}_{i}$ orthogonal to $u$.
		
		Let $u_{2}, u_{3}, u_{4} \in \mathbb{S}^{1}$ be unit vectors such that, for each $i\in\{1,2,3\}$, the counter-clockwise angle between $u_i$ and $u_{i+1}$ is $\frac{\pi}{4}$. By the pigeon-hole principle, among $u_{2}$, $u_{3}$ and $u_{4}$ there are two vectors with a common color $j \in \{ 1,2 \}$, so there are $3$ vectors among $u_1$, $u_2$, $u_3$ and $u_4$ with color $j$. We may assume that these $3$ vectors are $v_{1}, v_{2}, v_{3}$ and satisfy that $v_{1}$ and $v_{2}$ are orthogonal and the counter-clockwise angle between $v_{1}$ and $v_{3}$ is $\frac{\pi}{4}$. Let $l_{1}$, $l_{2}$ and $l_{3}$ be lines transversal to $\mathcal{F}_{j}$ orthogonal to $v_{1}$, $v_{2}$ and $v_{3}$, respectively.
		
		To prove the first part of Theorem \ref{thm:DolnikovCW}, suppose that $K$ has constant width $1$. Then every translate of $K$ is contained in a square of side $1$. We denote $\mathcal{G}_{j}$ as the family of squares of side $1$ that contains some set from $\mathcal{F}_{j}$. Since every set in $\mathcal{F}_{j}$ intersects the lines $l_{1}, l_{2}, l_{3}$, it follows that for every direction $v_{i}$, the centers of the squares in $\mathcal{G}_{j}$ are contained in the strip bounded by two  lines parallel to $l_{i}$. The width of these strips is $1$. Assume that $v_1=(1,0)$, then the centers of the squares in $\mathcal{G}_{j}$ are always contained in one of the two pentagons from Figure \ref{figure_regions-2colors}.
		
		\begin{figure}
			\includegraphics{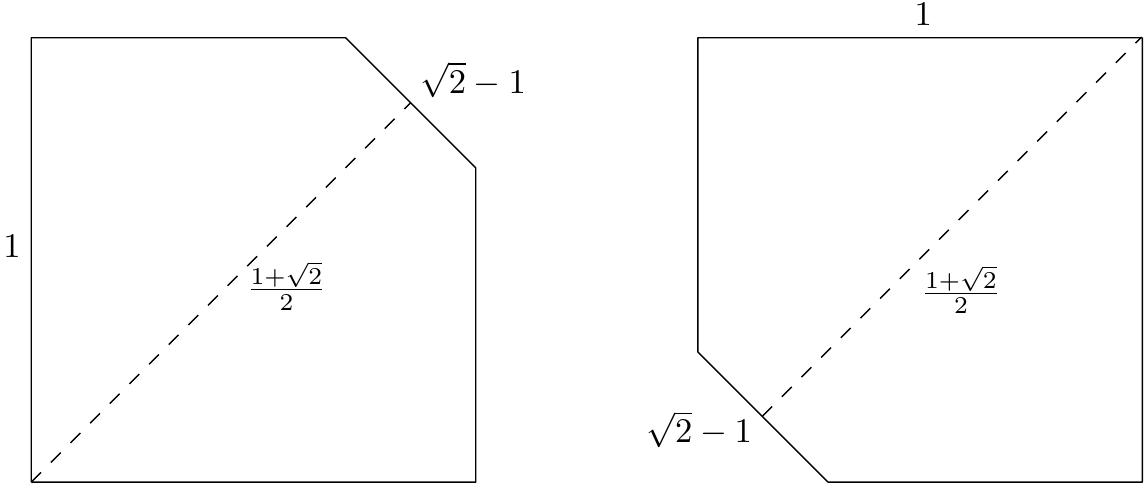}
			\caption{The region bounded by the parallel lines to $l_{1}, l_{2}, l_{3}$ is contained in one of the two pentagons.}
			\label{figure_regions-2colors}
		\end{figure}
		
		We claim that every one of the two pentagons can be covered by $3$ translates of $-K$. This would imply that the family $\mathcal{F}_{j}$ can be pierced by $3$ points. By symmetry we only have to show that one of the two pentagons from Figure \ref{figure_regions-2colors} can be covered by $3$ translates of $-K$. In order to do this, we label and divide the pentagon as in Figure \ref{figure_regions2-2colors}.
		Here $I$ is the midpoint of the segment $CD$. The points $F$, $G$ and $H$ are points inside the pentagon $ABCDE$ such that $F$ is in the segment $AE$, $G$ is in the segment $AB$ and $AGHF$ is a square of side $\frac{5}{8}$.
		
		\begin{figure}
			\includegraphics{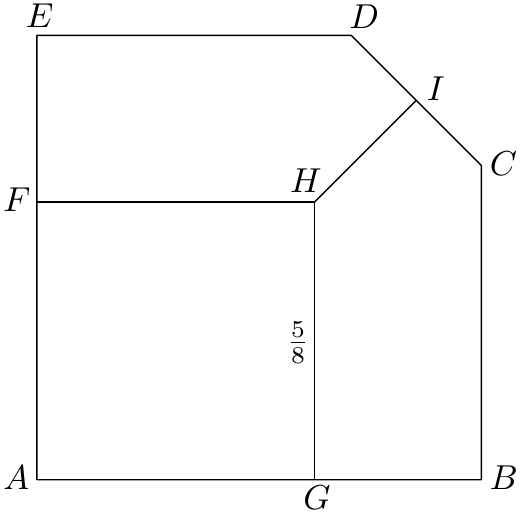}
			\caption{We divide the pentagon $ABCDE$ in a square of side $\frac{5}{8}$ and two pentagons.}
			\label{figure_regions2-2colors}
		\end{figure}
		
		Below we prove that the square $AGHF$ and the pentagons $GBCIH$ and $FHIDE$ can be covered by one translate of $-K$. By using a result of Chakerian \cite{Chakerian1969}, we only need to prove that every congruent copy of the square $AGHF$ and every congruent copy of the pentagon $GBCIH$ (or $FHIDE$) can be covered by a Reuleaux triangle of width $1$.
		Figure \ref{figure_ReuleauxS-2colors} shows how every rotation of the square $AGHF$ of side $\frac{5}{8}$ can be covered by a Reuleaux triangle of width $1$.
		%Let us calculate the sides of the pentagon $FHIDE$. Since $AI=\frac 12(1 + \sqrt{2})$ and $AH=\frac{5}{8} \sqrt{2}$, then $HI=AI-AH= \frac 12 - \frac 18 \sqrt{2}$. By Pythagoras theorem in the triangle $AIC$ we have that $AC=\frac 12 \sqrt{6}$, then by Pythagoras theorem in the triangle $ABC$ we have that $BC=\frac 12 \sqrt{2}=ED$. Thus the pentagon $FHIDE$ has sides $\frac{5}{8}$, $\frac{3}{8}$, $\frac 12 \sqrt{2}$, $\frac 12(\sqrt{2}-1)$ and $\frac12 - \frac 18\sqrt{2}$.
		Figure \ref{figure_ReuleauxP-2colors} shows how every rotation of the pentagon $FHIDE$ can be covered by a Reuleaux triangle of width $1$.
		
		\begin{figure}
			\includegraphics{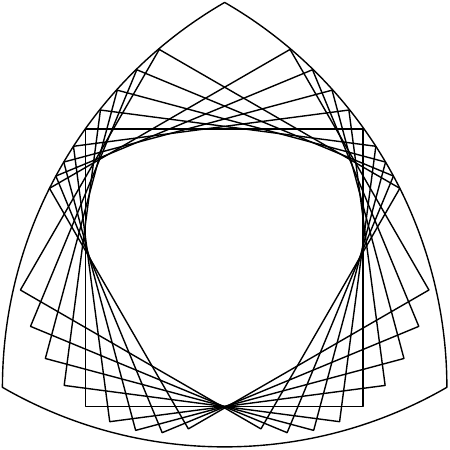}
			\caption{Every rotation of the square of side $\frac{5}{8}$ can be covered by a Reuleaux triangle of width $1$.}
			\label{figure_ReuleauxS-2colors}
		\end{figure}

		\begin{figure}
			\includegraphics{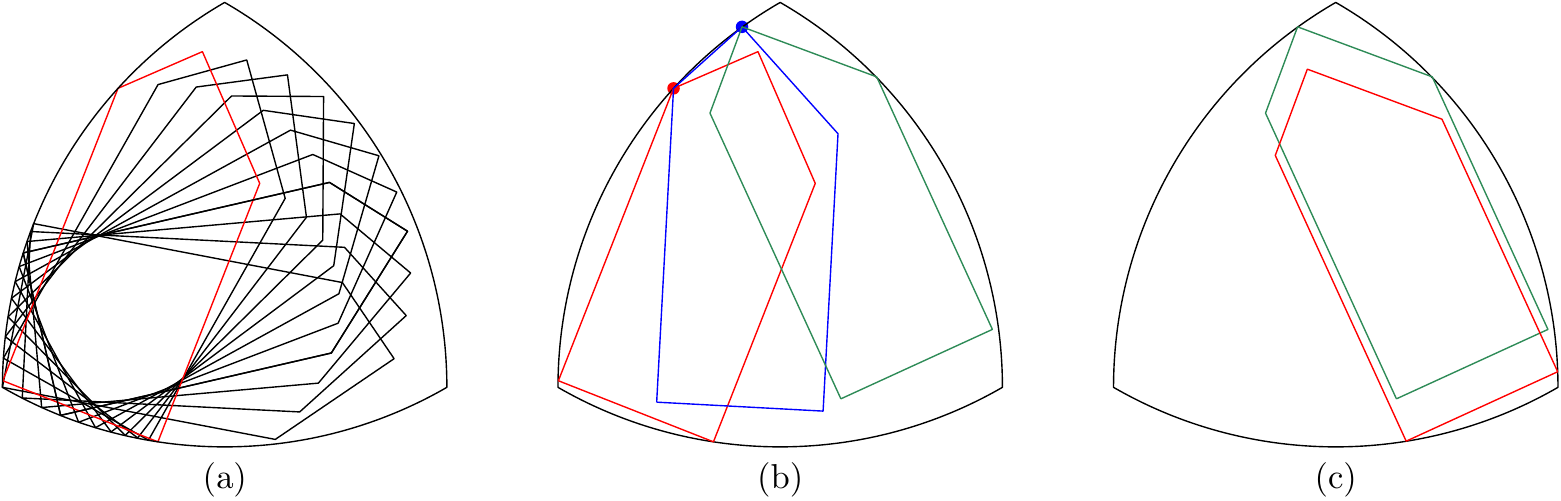}
			\caption{Every rotation of the pentagon can be covered by a Reuleaux triangle of width $1$. Indeed, in (a) we rotate the pentagon until the red pentagon, then in (b) we rotate the red pentagon with center in the red vertex until the blue pentagon, then we rotate the blue pentagon with center in the blue vertex until the green pentagon. Finally, in (c) we translate the green pentagon to the red pentagon and this red pentagon is a pentagon in (a) by rotating $\frac{2 \pi}{3}$.}
			\label{figure_ReuleauxP-2colors}
		\end{figure}
		
		For the second part of Theorem \ref{thm:DolnikovCW}, where the Banach-Mazur distance between $K$ and a disk is at most $1.1178$, we may assume (after a linear transformation) that $K$ contains a circle of radius $r = 1/2.2356 \approx 0.4473$ and is contained in a circle of radius $\frac 12$. Let $\mathcal{G}_{j}$ be the family of circles of radius $\frac 12$ that contain some set from $\mathcal{F}_{j}$ and let $\mathcal{H}_{j}$ be the family of circles of radius $r$ that are contained in some set from $\mathcal{F}_{j}$. Since every set in $\mathcal{F}_{j}$ intersects the lines $l_{1}, l_{2}, l_{3}$, it follows that for every direction $v_{i}$, the centers of the circles in $\mathcal{G}_{j}$ are contained in the strip bounded by two parallel lines to $l_{i}$. The width of these strips is $1$. Assume, without loss of generality, that $v_1=(1,0)$, then the centers of the circles in $\mathcal{G}_{j}$ are contained in one of the two pentagons from Figure \ref{figure_regions-2colors}.
		
		Suppose that the centers of the circles in $\mathcal{G}_{j}$ are contained in the pentagon $ABCDE$ shown in Figure \ref{figure_regions2-2colors}. This pentagon can be covered with $3$ circles of radius at most $r$ as shown in Figure \ref{figure_3circles-2colorsBM}. Here $I$ is the midpoint of the segment $CD$, $M$ and $N$ are the points in the segments $AE$ and $AB$ such that $AM = AN$ and $MN = 2 r$. From here a simple calculation gives that the circumradii of the triangles $MIE$, $ANM$ and $NBI$ are all at most $r$. This implies that the family $\mathcal{F}_{j}$ can be pierced by $3$ points.
		%By the Pythagorean theorem on the triangle $AMN$ we have that $AM \approx 0.6325 $ and therefore $ME = NB \approx 0.3675$.
	\end{proof}
		
		\begin{figure}
			
			\includegraphics{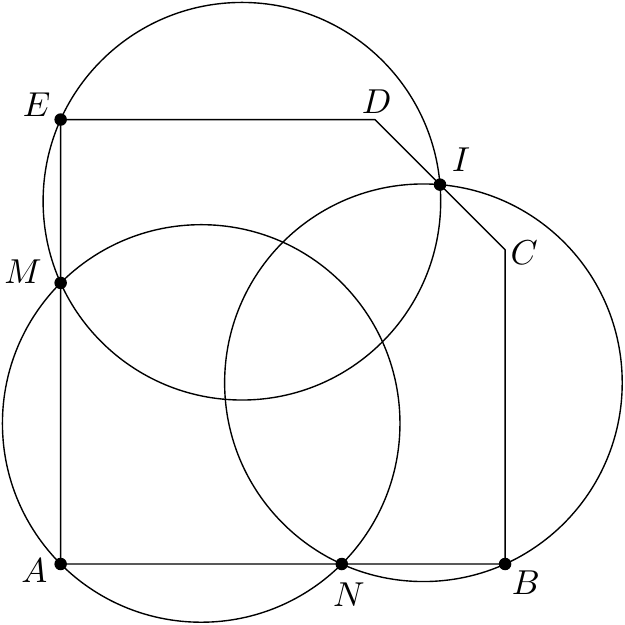}
			\caption{The pentagon can be covered with $3$ circles of radius at most $0.4474$.}
			\label{figure_3circles-2colorsBM}
		\end{figure}
		
		%We will prove that the circunradius of the triangles $ MEI $ and $ NBI $ is less than $ r_{1} $. By the law of cosines in the triangle $ANI$ we obtain that $ NI \leq 0.8816 $. Since $ ABCI $ is a cyclic quadrilateral, then $ \cos (IBC) = \cos (IAC) $. In triangle $ IAC $ we have that $ \cos (IAC) = \frac{AI}{AC} = \frac{\frac{\sqrt{2}+1}{2}}{\sqrt{\frac{3}{2}}} = \frac{\sqrt{2}+1}{\sqrt{6}} $. Then $ \sin (NBI) = \cos (IBC) = \cos (IAC) = \frac{\sqrt{2}+1}{\sqrt{6}} $. Thus the circunradius of the triangle $ NBI $ is $ \frac{NI}{2 \sin (NBI)} \leq \frac{0.8816 \sqrt{6}}{2 (\sqrt{2}+1)} \leq 0.4473 \leq r_{1} $, and analogously the circunradius of the triangle $ MEI $ is at most $ r_{1} $. In addition, by construction we know that the circunradius of the triangle $ AMN $ is $ r_{1} $. 
	
	%By the same argument given in the end of the last subsection we only need our Lemma \ref{lem:specialVCH} to prove the theorems of this subsection (in particular we can avoid Theorem \ref{thm:Karasev}).
	
	\section{Conclusions and future work}
	
	In Theorem \ref{thm:PlaneVCH} we gave the best numbers for the $2$-dimensional case of Theorem \ref{thm:VCH}. An interesting question is what are the best numbers in this Theorem for any given dimension. In particular, does Theorem \ref{thm:VCH} hold with $f(d) = 1$ and $g(d) = \lceil \frac{d+1}{2} \rceil$? Theorem \ref{thm:PlaneVCH} shows that this is true for $d=2$.
	
	Another question is how much can the numbers in Theorems \ref{thm:DolnikovNew} and \ref{thm:SpecialVCH} be improved. We do not believe that they are optimal.
	
	On the other hand, we use Lemma \ref{lem:specialVCH} to connect colorful theorems with monochromatic ones. For example, we used it to prove that Theorems \ref{thm:Karasev}, \ref{thm:Grunbaum}, \ref{thm:Danzer} and \ref{thm:rectangles} imply parts (a), (b), (c) and (d) of Theorem \ref{thm:SpecialVCH}, respectively. Perhaps similar techniques could be useful in order to prove the following colorful conjecture of Theorem \ref{thm:Danzer}.
	
	\begin{conjecture}
		Let $\mathcal{F}_{1}, \mathcal{F}_{2}$ be finite families of circles in $\mathbb{R}^{2}$. Suppose that $A \cap B \neq \emptyset$ for every $A \in \mathcal{F}_{1}$ and $B \in \mathcal{F}_{2}$. Then either $\mathcal{F}_{1}$ or $\mathcal{F}_{2}$ can be pierced by $4$ points.
	\end{conjecture}
	
	\section{Acknowledgments}
	This work was supported by UNAM-PAPIIT IN111923.
	
	\bibliographystyle{amsalpha}
	\bibliography{referencias}

\newcommand{\etalchar}[1]{$^{#1}$}
\providecommand{\bysame}{\leavevmode\hbox to3em{\hrulefill}\thinspace}
\providecommand{\MR}{\relax\ifhmode\unskip\space\fi MR }
\providecommand{\MRhref}[2]{%
  \href{http://www.ams.org/mathscinet-getitem?mr=#1}{#2}
}
\providecommand{\href}[2]{#2}
\begin{thebibliography}{MSRPR20}

\bibitem[ABB{\etalchar{+}}09]{Arocha2009}
J.~L. Arocha, I.~Bárány, J.~Bracho, R.~Fabila, and L.~Montejano, \emph{Very
  {C}olorful {T}heorems}, Discrete Comput. Geom. \textbf{42} (2009), no.~2,
  142--154.

\bibitem[Bá82]{Barany1982}
I.~Bárány, \emph{A generalization of {C}arathéodory theorem}, Discrete Math.
  \textbf{40} (1982), 141--152.

\bibitem[Cha69]{Chakerian1969}
G.D. Chakerian, \emph{Intersection and covering properties of convex sets}, The
  American Mathematical Monthly \textbf{76} (1969), no.~7, 753--766.

\bibitem[Dan86]{Danzer1986}
L.~Danzer, \emph{Zur {L}{\"o}sung des {G}allaischen {P}roblems {\"u}ber
  {K}reisscheiben in der euklidischen {E}bene}, Studia Sci. Math. Hungar.
  \textbf{21} (1986), 111--134.

\bibitem[DJ11]{Dumitrescu2011}
A.~Dumitrescu and M.~Jiang, \emph{Piercing translates and homothets of a convex
  body}, Algorithmica \textbf{61} (2011), 94--115.

\bibitem[FDFK93]{Fon1993}
D.G. Fon-Der-Flaass and A.~V. Kostochka, \emph{Covering boxes by points},
  Discrete mathematics \textbf{120} (1993), no.~1-3, 269--275.

\bibitem[GN22]{Gomez-Navarro2022}
C.~Gomez-Navarro, \emph{Colorful theorems in discrete and convex geometry},
  Master's thesis, Universidad Nacional Aut{\'o}noma de M{\'e}xico, 2022.

\bibitem[Gr{\"u}59]{Grunbaum1959}
B.~Gr{\"u}nbaum, \emph{On intersections of similar sets}, Portugal Math.
  \textbf{18} (1959), 155--164.

\bibitem[HD66]{Hadwiger1966}
H.~Hadwiger and H.~Debrunner, \emph{Combinatorial {G}eometry in the plane},
  Dover Publications, 1966.

\bibitem[Hel23]{Helly1923}
E.~Helly, \emph{Uber {M}engen konvexer {K}orper mit gemeinschaftlichen
  {P}unkten}, Jahres-berichte der Deutschen Math.-Verein. \textbf{32} (1923),
  175--176.

\bibitem[HPT08]{Holmsen2008}
A.~F. Holmsen, J.~Pach, and H.~Tverberg, \emph{Points sorrounding the origin},
  Combinatorica \textbf{28} (2008), no.~6, 633--644.

\bibitem[JCMS15]{Jeronimo2015}
J.~Jer{\'o}nimo-Castro, A.~Magazinov, and P.~Sober{\'o}n, \emph{On a problem by
  {D}ol’nikov}, Discrete Mathematics \textbf{338} (2015), no.~9, 1577--1585.

\bibitem[Joh48]{John1948}
F.~John, \emph{Extremum problems with inequalities as subsidiary conditions},
  Courant Anniversary Volume (1948), 187--204.

\bibitem[Kar00]{Karasev2000}
R.~N. Karasev, \emph{Transversals for families of translates of a
  two-dimensional convex compact set}, Discrete Comput. Geom. \textbf{24}
  (2000), no.~2, 345--354.

\bibitem[Kar08]{Karasev2008}
\bysame, \emph{Piercing {F}amilies of {C}onvex {S}ets with the $ d
  $-{I}ntersection {P}roperty in $ \mathbb{R}^{d} $}, Discrete Comput. Geom.
  \textbf{39} (2008), no.~4, 766--777.

\bibitem[KKM29]{Knaster1929}
B.~Knaster, C.~Kuratowski, and S.~Mazurkiewicz, \emph{Ein {B}eweis des
  {F}ixpunktsatzes f{\"u}r n-dimensionale {S}implexe}, Fundamenta Mathematicae
  \textbf{14} (1929), no.~1, 132--137.

\bibitem[KNPS06]{Kim2006}
S.-J. Kim, K.~Nakprasit, M.~J. Pelsmajer, and J.~Skokan, \emph{Transversal
  numbers of translates of a convex body}, Discrete Mathematics \textbf{306}
  (2006), no.~18, 2166--2173.

\bibitem[MMO19]{Martini2019}
H.~Martini, L.~Montejano, and D.~Oliveros, \emph{Bodies of constant width},
  Springer, 2019.

\bibitem[MSRPR20]{Martinez2020}
L.~Martínez-Sandoval, E.~Roldán-Pensado, and N.~Rubin, \emph{Further
  {C}onsequences of the {C}olorful {H}elly {H}ypothesis}, Discrete Comput.
  Geom. \textbf{63} (2020), no.~4, 848--866.

\bibitem[MZ21]{Mcginnis2021}
D.~McGinnis and S.~Zerbib, \emph{Line transversals in families of connected
  sets the plane}, arXiv preprint arXiv:2103.05565 (2021).

\bibitem[NT10]{Naszodi2010}
M.~Nasz{\'o}di and S.~Taschuk, \emph{On the transversal number and
  {V}{C}-dimension of families of positive homothets of a convex body},
  Discrete mathematics \textbf{310} (2010), no.~1, 77--82.

\bibitem[Rog63]{Rogers1963}
C.A. Rogers, \emph{Covering a sphere with spheres}, Mathematika \textbf{10}
  (1963), no.~2, 157--164.

\bibitem[VG05]{Verger2005}
J.L. Verger-Gaugry, \emph{Covering a ball with smaller equal balls in $
  \mathbb{R}^{n} $}, Discrete and Computational Geometry \textbf{33} (2005),
  143--155.

\end{thebibliography}
	
\end{document}